\documentclass[11pt, reqno]{amsart}
\usepackage{lmodern}
\usepackage{amsmath, amsthm, amssymb, amsfonts}
\usepackage[normalem]{ulem}
\usepackage{hyperref}

\usepackage{verbatim}
\usepackage{caption}
\usepackage{mathtools}

\usepackage{tikz-cd}
\theoremstyle{plain}
\newtheorem{thm}{Theorem}[section]
\newtheorem{cor}[thm]{Corollary}
\newtheorem{lem}[thm]{Lemma}
\newtheorem{lemma}[thm]{Lemma}
\newtheorem{prop}[thm]{Proposition}

\theoremstyle{definition}

\theoremstyle{remark}
\newtheorem{rmk}[thm]{Remark}

\newcommand{\BC}{{\mathbb{C}}}

\newcommand{\BG}{{\mathbb{G}}}
\newcommand{\BH}{{\mathbb{H}}}

\newcommand{\BP}{{\mathbb{P}}}
\newcommand{\BQ}{{\mathbb{Q}}}

\newcommand{\BZ}{{\mathbb{Z}}}

\newcommand{\CC}{{\mathcal C}}

\newcommand{\CH}{{\mathcal H}}
\newcommand{\CI}{{\mathcal I}}
\newcommand{\CJ}{{\mathcal J}}

\newcommand{\CN}{{\mathcal N}}
\newcommand{\CO}{{\mathcal O}}

\newcommand{\CQ}{{\mathcal Q}}

\newcommand{\CU}{{\mathcal U}}
\newcommand{\CV}{{\mathcal V}}

\newcommand{\CX}{{\mathcal X}}

\DeclareMathOperator{\Hilb}{Hilb}

\DeclareFontFamily{OT1}{rsfs}{}
\DeclareFontShape{OT1}{rsfs}{n}{it}{<-> rsfs10}{}
\DeclareMathAlphabet{\curly}{OT1}{rsfs}{n}{it}

\newcommand{\p}{\mathbb{P}}

\newcommand{\Mbar}{{\overline M}}

\newcommand{\GW}{\mathsf{GW}}
\newcommand{\Sym}{{\mathrm{Sym}}}
\newcommand{\Z}{\mathsf{Z}}
\newcommand{\ev}{{\mathrm{ev}}}

\begin{document}
\title[Rational curves in holomorphic symplectic varieties]
{Rational curves in holomorphic symplectic varieties and Gromov--Witten invariants}
\date{\today}

\author{Georg Oberdieck}
\address{Universit\"at Bonn, Mathematisches Institut}
\email{georgo@math.uni-bonn.de}

\author{Junliang Shen}
\address{Massachusetts Institute of Technology, Department of Mathematics}
\email{jlshen@mit.edu}

\author{Qizheng Yin}
\address{Peking University, Beijing International Center for Mathematical Research}
\email{qizheng@math.pku.edu.cn}

\begin{abstract}
We use Gromov--Witten theory to study rational curves in holomorphic symplectic varieties. 
We present a numerical criterion for the existence of uniruled divisors swept out by rational curves in the primitive curve class of a very general holomorphic symplectic variety of $K3^{[n]}$ type. 
We also classify all rational curves in the primitive curve class of the Fano variety of lines in a very general cubic~$4$-fold. % and prove the irreducibility of the corresponding moduli space.
Our proofs rely on Gromov--Witten calculations by the first author,
and in the Fano case on a geometric construction of Voisin.
In the Fano case a second proof via classical geometry is sketched.

\end{abstract}
\baselineskip=14.5pt
%\baselineskip=12.5pt
%Standard baseline skip: 1.2*fontsize; see \the\baselineskip
\maketitle

\setcounter{tocdepth}{1} 

\tableofcontents
\setcounter{section}{-1}

\section{Introduction}

\subsection{Overview}
Rational curves in $K3$ surfaces have been investigated
for decades from various angles.
In contrast,
%However, 
not much is known about the geometry of rational curves in the higher-dimensional analogs of $K3$ surfaces---holomorphic symplectic varieties.\footnote{A nonsingular projective variety $X$ is holomorphic symplectic 
if it is simply connected and $H^0(X, \Omega^2_X)$ is generated by a nowhere degenerate holomorphic $2$-form.}
In this paper, we use \emph{Gromov--Witten theory}
(intersection theory of the moduli space of stable maps)
together with classical methods to study these rational curves.
%We observe that the geometry of rational curves in higher-dimensional holomorphic symplectic varieties is different in several aspects.

\subsection{Rational curves}
Let $(X, H)$ be a very general polarized holomorphic symplectic variety of dimension $2n$, and let $\beta \in H_2(X, \BZ)$ be the primitive curve class.
The moduli space $\Mbar_{0,m}(X,\beta)$ of genus~$0$ and $m$-pointed stable maps to $X$ in class $\beta$
%, denoted by $\Mbar_{0,m}(X,\beta)$, 
is pure of expected dimension $2n-2+m$; see Proposition~\ref{prop2.1}. 
Consider the decomposition
%Let $(X, H)$ be a very general polarized holomorphic symplectic variety of dimension $2n$, and let $\beta \in H_2(X, \BZ)$ be the primitive curve class.
%Let $\Mbar_{0,m}(X,\beta)$ be the moduli space parametrizing genus~$0$ and $m$-pointed stable maps $f: C \rightarrow X$ in curve class
%\[ f_\ast [C] = \beta \in H_2(X, \mathbb{Z}). \]
%The moduli space $\Mbar_{0,0}(X, \beta)$ is pure of the expected dimension~$2n - 2$. The universal family over the moduli space admits a decomposition
\begin{equation}\label{decomp1}
\Mbar_{0,1}(X,\beta) = M^0 \cup M^1 \cup \cdots \cup M^{n - 1}
\end{equation}
such that the general fibers of the restricted evaluation map
\[\mathrm{ev}: M^i \rightarrow \mathrm{ev}(M^i) \subset X\]
are of dimension $i$.
%Since every $M^i$ is of pure dimension $2n-1$,
The image of $M^0$ under $\mathrm{ev}$ is precisely the
union of all uniruled divisors swept out by rational curves in class $\beta$.
More generally, the image $\mathrm{ev}(M^i)$ is the codimension $i+1$ locus of points on $X$ through which
passes
an $i$-dimensional family of rational curves in class $\beta$. %, et cetera.

In \cite[Conjecture 4.3]{MoPa2}, Mongardi and Pacienza conjectured that for all $i$
\begin{equation*}\label{MPC}
M^i \neq \emptyset,
\end{equation*}
which would imply the existence of algebraically coisotropic subvarieties in~$X$ in the sense of Voisin \cite{V2}.

In Theorems~\ref{A3} and \ref{mainthm} below, we provide counterexamples to this conjecture
which illustrate ``pathologies" of rational curves in higher-dimensional holomorphic symplectic varieties.
Two typical examples are as follows.
%Recall that a variety is of $K3^{[n]}$ type if it is deformation equivalent to the Hilbert scheme of $n$ points on a $K3$ surface.
% below provide the following two typical examples:
\begin{enumerate}
\item[(i)] There exist a very general pair $(X, H)$ of~$K3^{[8]}$ type with $M^0 = \emptyset$.
In other words, on $(X,H)$ there exists no uniruled divisor swept out by rational curves in the primitive class $\beta$.
\item[(ii)] For the Fano variety of lines in a very general cubic $4$-fold, we have~$M^1 = \emptyset$.
\end{enumerate}
Here a variety is of $K3^{[n]}$ type if it is deformation equivalent to the Hilbert scheme of $n$ points on a $K3$ surface.
%We discuss (i) and (ii) with more details in the next two subsections.

%\footnote{A variety is of $K3^{[n]}$ type if it is deformation equivalent to the Hilbert scheme of $n$ points on a $K3$ surface.}

%\begin{enumerate}
%\item[(i)] In the $K3^{[n]}$ case, the class
%\[\mathrm{ev}_\ast[M^0] \in H^2(X, \BQ)\]
%is computed explicitly in \cite{Ob2}; see also Appendix~\ref{appuni}. In particular, Corollary \ref{A3} implies $M^0 \neq \emptyset$ for $K3^{[2]}$. On the other hand, Corollary \ref{A3} also provides an example of $K3^{[8]}$ type satisfying $M^0 = \emptyset$.
%\item[(ii)] Theorem \ref{mainthm} implies $M^1 = \emptyset$ for a very general holomorphic symplectic variety of $K3^{[2]}$ type of degree $6$ (and divisibility $2$).
%\end{enumerate}

\subsection{Uniruled divisors}
On a holomorphic symplectic variety $X$, let
\begin{equation*} 
( - , - ) : H_2(X,\BZ) \times H_2(X,\BZ) \to \BQ
\end{equation*}
denote the unique $\BQ$-valued extension of the Beauville--Bogomolov form on~$H^2(X,\BZ)$.
\begin{comment}
If $X$ is of $K3^{[n]}$ type with $n \geq 2$ and $\beta \in H_2(X,\BZ)$ is primitive
the \emph{residue} of $\beta$ is the integer $r(\beta)$ such that
\[ \langle H_2(X,\BZ), \beta \rangle = \frac{r(\beta)}{2n-2} \BZ. \]
While $r(\beta)$ is defined only up to multiplication by $\pm 1$, the residue set
\[ \pm [\beta] = \{ \pm r(\beta) \} \subset \BZ / (2n-2) \BZ \]
is independent of choices. If $n=1$, we set $\pm [\beta]=0$.
\end{comment}
If $X$ is of $K3^{[n]}$ type and $n \geq 2$, there is an isomorphism of abelian~groups
\[ r : H_2(X,\BZ)/H^2(X,\BZ) \to \BZ / (2n-2) \BZ, \]
unique up to multiplication by $\pm 1$,
such that $r(\alpha) = 1$ for some $\alpha \in H_2(X, \BZ)$ with $(\alpha, \alpha) = \frac{1}{2-2n}$.
%The morphism~$r$ is unique up to multiplication by $\pm 1$. 
Given a class $\beta \in H_2(X,\BZ)$, we define its residue set by
\begin{equation*} 
\pm [\beta] = \{ \pm r(\beta) \} \subset \BZ / (2n-2) \BZ. \end{equation*}
In case $n=1$, we set $\pm [\beta]=0$.

The following theorem provides a complete numerical criterion for the existence of uniruled divisors swept out by rational curves in the primitive class of a very general variety of $K3^{[n]}$ type.

\begin{thm}\label{A3}
Let $X$ be a holomorphic symplectic variety of $K3^{[n]}$ type, and let $\beta \in H_2(X,\BZ)$ be a primitive curve class. If
\begin{gather*}
(\beta, \beta) = -2 + \sum_{i = 1}^{n - 1} 2 d_i - \frac{1}{2n - 2}\left(\sum_{i = 1}^{n - 1} r_i\right)^2, \\
\pm[\beta] = \pm\left[\sum_{i = 1}^{n - 1} r_i\right]
\end{gather*}
for some $d_i, r_i \in \BZ$ satisfying $2d_i - \frac{r_i^2}{2} \geq 0$, %, $i = 1, \ldots, n - 1$,
then there exists a uniruled divisor on $X$ swept out by rational curves in class $\beta$.
%Here $(\beta, \beta)$ and $\pm[\beta]$ stand for the Beauville--Bogomolov norm and the residue set of $\beta$.\footnote{See Section \ref{Subsection_curve_classes} for more details.} 
The converse holds if~$\beta$ is irreducible.
\end{thm}

For a very general pair $(X, \beta)$ with $X$ of $K3^{[n]}$ type and $\beta$ the primitive curve class, Theorem \ref{A3} implies that
\begin{enumerate}
\item[(i)] $M^0 \neq \emptyset$ when $n \leq 7$, and
\item[(ii)] for every $n\geq 8$, there exists $(X, \beta)$ such that $M^0= \emptyset$.
\end{enumerate}
The first instance of case (ii) is given by a very general pair $(X,\beta)$ of $K3^{[8]}$ type with
$(\beta,\beta) = \frac{3}{14}$ and $\pm [\beta] = \pm [5]$.\footnote{Such a pair $(X,\beta)$ can be obtained by deforming $(\Hilb^8(\mathsf{S}), \beta')$, where $\mathsf{S}$ is a $K3$ surface of genus $2$ with polarization $\mathsf{H}$ and $\beta' = \mathsf{H} + 5 \mathsf{A}$
with $\mathsf{A}$ the exceptional curve class; see Section~\ref{GW_Appendix} for the notation.}

\subsection{Fano varieties of lines}
Let $Y \subset \BP^5$ be a nonsingular cubic $4$-fold. 
By Beauville and Donagi \cite{BD}, the Fano variety of lines in~$Y$
\[ F = \{ l \in \mathrm{Gr}(2,6) : l \subset Y \} \]
is a holomorphic symplectic $4$-fold.
These varieties form a $20$-dimensional family of
polarized holomorphic symplectic varieties 
of $K3^{[2]}$ type.
%of degree $6$
%with respect to the Beauville--Bogomolov~form.

In \cite{V0}, Voisin constructed a rational self-map
\begin{equation}\label{rationalmap}
\varphi: F \dashrightarrow F
\end{equation}
sending a general line $l$ to its residual line with respect to the unique plane~$\BP^2 \subset \BP^5$ tangent to $Y$ along $l$.
When $Y$ is very general,
the exceptional divisor associated to the resolution of $\varphi$
\begin{equation}\label{uniruled}
\begin{tikzcd}
D  = \BP(\CN_{S/F})\arrow{r}{\phi} \arrow{d}{p} & F \\
S
\end{tikzcd} 
\end{equation}
is a $\p^1$-bundle over a nonsingular surface $S \subset F$; see Amerik \cite{A}.
The image of each fiber
\begin{equation*}
\phi(p^{-1}(s)) \subset F, \quad s \in S \label{dgfsdgsdf}
\end{equation*}
is a rational curve lying in the primitive curve class in $H_2(F, \BZ)$.

The following theorem shows that every rational curve in the primitive curve class is of this form in a unique way.

%The following theorem shows that every rational curve in the primitive curve class is of this form.

%Hence $\phi : D \to F$ is a uniruled divisor swept out by rational curves in the primitive curve class.
%\[ \phi : p^{-1}(s) = \p^1 \to F \]
%The following theorem describes all rational curves in the primitive curve class of $F$.

\begin{thm}\label{mainthm}
Let $F$ be the Fano variety of lines in a very general cubic~$4$-fold.
Then for every rational curve $C \subset F$ in the primitive curve class, there is a unique $s \in S$ such that $C = \phi(p^{-1}(s))$.
%Then every rational curve in the primitive curve class is the image under $\phi$ of a unique fiber of $p: D \to S$.
%Then every rational curve in the primitive curve class is given by a unique fiber of $p: D \to S$.
\end{thm}

%Theorem \ref{mainthm} implies that $M^1 = \emptyset$ in the decomposition (\ref{decomp1}) for $F$, which is a counter-example to (\ref{MPC}) when $X$ is of $K3^{[2]}$ type and $i=1$. 

By a result of Huybrechts \cite[Section 6]{Huy}, the surface $S$ is connected.\footnote{
In an earlier version of this paper the connectedness of $S$ was claimed.
We thank Daniel Huybrechts for pointing out a gap in our earlier argument and presenting a new proof.
The connectedness of $S$ is not needed in the proof of Theorem~\ref{mainthm}.}
%For a very general $F$, Theorem \ref{mainthm} implies that $M^1 = \emptyset$ in the decomposition (\ref{decomp1}). 
%We also show that $S$ is connected and calculate its first Chern class; see Corollary \ref{cor1.3}. 
In particular, the moduli space of rational curves in the primitive curve class of a very general $F$ is irreducible. This implies $M^1 = \emptyset$ in the decomposition~(\ref{decomp1}) and the following.

%The image $\phi(D)$ is a uniruled divisor swept out by rational curves in the primitive curve class.
%We obtain a uniqueness statement for uniruled divisors in $F$.
%We obtain that the uniruled divisor $\phi(D)$
%Since $S$ is connected (by Corollary\ref{cor1.3}) we obtain
%We also show that the uniruled divisor $\phi(D)$ is 

%\begin{cor}\label{maincor}
%For a very general $F$, all rational curves in the primitive curve class sweep out an irreducible uniruled divisor.
%\end{cor}

\begin{cor}\label{maincor}
For a very general $F$, there is a unique irreducible uniruled divisor swept out by rational curves in the primitive curve class.
\end{cor}

The moduli space of rational curves in the primitive curve class of a very general $K3$ surface always has more than one irreducible component. Corollary \ref{maincor} indicates a difference between rational curves in $K3$ surfaces and in higher-dimensional holomorphic symplectic varieties.

\subsection{Idea of proofs}
We briefly explain how Gromov--Witten theory \cite{FP} controls rational curves in the primitive class $\beta$ of a very general polarized holomorphic symplectic variety $(X, H)$ of $K3^{[n]}$ type.

Since the evaluation map $\mathrm{ev}$ is generically finite on the component $M^0$
but contracts positive dimensional fibers on all other components in the decomposition (\ref{decomp1}),
the (non)emptiness of $M^0$ is detected by the pushforward
\begin{equation} \mathrm{ev}_\ast [\Mbar_{0,1}(X,\beta)] \in H^2(X,\BQ). \label{sfdfsd} \end{equation}
For the Fano variety of lines $X=F$, a key observation is that the emptiness of~$M^1$ can be further detected
by the Gromov--Witten correspondence
\begin{equation} \mathrm{ev_{12}}_\ast [\Mbar_{0,2}(X,\beta)] \in H^{4n}(X\times X, \BQ). \label{3sdf} \end{equation}
The class \eqref{3sdf} has contributions from all of the components in (\ref{decomp1}), and contains strictly more information than the $1$-pointed class \eqref{sfdfsd}.

%Let $\Mbar_{0,m}(X,\beta)$ be the moduli space of genus~$0$ and $m$-pointed stable maps $f: C \rightarrow X$ in curve class
%\[f_\ast [C] = \beta \in H_2(X, \mathbb{Z}).\]
%Concerning the decomposition (\ref{decomp1}), the key observation is that the nonemptiness of $M^1$ for $X= F$ and the nonemptiness of $M^0$ for any $X$ are detected by the push-forward of the fundamental classes
%\begin{equation}\label{GW12}
%\begin{gathered}
%\mathrm{ev}_\ast [\Mbar_{0,1}(X,\beta)] \in H^2(X,\BQ),\\
%\mathrm{ev_{12}}_\ast [\Mbar_{0,2}(F,\beta)] \in H^8(F\times F, \BQ),
%\end{gathered}
%\end{equation}
%where $\mathrm{ev}$ and $\mathrm{ev_{12}}$ are the evaluation maps. 
Since $\Mbar_{0,m}(X,\beta)$ is pure of the expected dimension, its fundamental class coincides with the (reduced) virtual fundamental class \cite{BF, LT},
\[
[\Mbar_{0,m}(X,\beta)] = [\Mbar_{0,m}(X,\beta)]^{\mathrm{vir}}.
\]
Hence the classes \eqref{sfdfsd} and \eqref{3sdf} are determined by the Gromov--Witten invariants of $X$. By deformation invariance, the Gromov--Witten invariants can be calculated on a special model given by the Hilbert scheme of points of an elliptic $K3$ surface; see \cite{Ob2} and Section \ref{GW_Appendix}.

Our proofs of Theorems \ref{A3} and \ref{mainthm} are intersection-theoretic. In Appendix~\ref{Classical}, we also sketch an alternative proof of Theorem \ref{mainthm} using a series of classification results in classical projective geometry.\footnote{The proof in Appendix \ref{Classical} was found only after
a first version of this article appeared online.
While Theorem~\ref{mainthm} can be proven classically, the quantitative information obtained from Gromov--Witten theory was essential for us to find the statement.}

\subsection{Conventions}
We work over the complex numbers. A statement holds for a \emph{very general} polarized projective variety $(X,H)$ if it holds away from a countable union of proper Zariski-closed subsets in the corresponding component of the moduli space.

\subsection{Acknowledgements}
The project was started when all three authors were participating in the
Spring 2018 workshop \emph{Enumerative Geometry Beyond Numbers}
at the \emph{Mathematical Sciences Research Institute} (MSRI) in Berkeley.
We thank the institute and the organizers for providing excellent working conditions.
We would also like to thank Olivier Debarre, Jun~Li, Eyal Markman, Rahul Pandharipande, Mingmin Shen, Zhiyu Tian, and Chenyang Xu for useful discussions,
and the referee for insightful comments.

G.~O.~was supported by the National Science Foundation Grant DMS-1440140 while in residence at MSRI, Berkeley.
J.~S.~was supported by the grant ERC-2012-AdG-320368-MCSK in the group of Rahul Pandharipande at ETH Z\"urich. Q.~Y.~was supported by the NSFC Young Scientists Fund 11701014.

\section{Moduli spaces of stable maps} \label{sec:2}
We discuss properties of the moduli spaces of stable maps to holomorphic symplectic varieties, and introduce tools from Gromov--Witten theory.

\subsection{Dimensions} \label{Subsection_Dimensions}
Let $X$ be a holomorphic symplectic variety of dimension~$2n$, and let $\beta \in H_2(X, \mathbb{Z})$ be an \emph{irreducible} curve class. We show that the moduli space $\Mbar_{0,1}(X,\beta)$ of genus $0$ pointed stable maps to $X$ in class~$\beta$ is pure of the expected dimension.

Let $M$ be an irreducible component of $\Mbar_{0,1}(X,\beta)$. We know \emph{a priori}
\[
\dim M \geq \int_\beta c_1(X) + \dim X -1 = 2n-1.
\]
Consider
the restriction of the evaluation map to $M$,
\begin{equation}\label{resev}
\mathrm{ev}: M \rightarrow Z = \mathrm{ev}(M) \subset X.
\end{equation}

\begin{prop}\label{prop2.1}
If a general fiber of \eqref{resev} is of dimension~$r-1$, then
\begin{enumerate}
\item[(i)] $\dim Z = 2n-r$, so that $\dim M = 2n - 1$;
\item[(ii)] $r \leq n$;
\item[(iii)] a general fiber of the MRC fibration\footnote{We refer to \cite{GHS} for the definition and properties of the maximal rationally connected (MRC) fibration.} $Z \dashrightarrow B$ is of dimension $r$.
\end{enumerate}
\end{prop}

\begin{proof}
Since the curve class $\beta$ is irreducible, the family of rational curves $M \rightarrow T \subset \Mbar_{0,0}(X, \beta)$ viewed as in $X$ is unsplit in the sense of \cite[IV, Definition 2.1]{Kol}. Given a general point $x \in Z$, let $T_{x} \subset T$ be the Zariski-closed subset parametrizing maps passing through $x$. Consider the universal family $ \CC_x \to T_{x}$ and the restricted evaluation map
\begin{equation*}\label{vx}
\mathrm{ev}: \CC_x \to V_{x} = \mathrm{ev}(\CC_{x}) \subset Z.
\end{equation*}
By \cite[IV, Proposition 2.5]{Kol}, we have
\[\dim T = \dim Z + \dim V_{x} - 2.\]
Hence $\dim V_x = \dim M - \dim Z + 1 = r$. In other words, rational curves through a general point of $Z$ cover a Zariski-closed subset of dimension $r$.

A general fiber of the MRC fibration $Z \dashrightarrow B$ is thus of dimension~$\geq r$. By an argument of Mumford (see \cite[Lemma 1.1]{V2}), this implies $\dim Z \leq 2n-r$ and $r \leq n$. On the other hand, since $\dim M \geq 2n - 1$, we have \[\dim Z = \dim M - (r-1) \geq 2n-r.\]
Hence there is equality $\dim Z = 2n-r$, and the dimension of a general fiber of $Z \dashrightarrow B$ is exactly $r$.
\end{proof}

\begin{comment}
By the definition and properties of the MRC fibration (see \cite{GHS}), we know that a general fiber of $Z_j \dashrightarrow B_j$ has dimension exactly $r$. The fiber is also birational to $V_{x_j}$ for some general $x_j$. By construction, two general points on a general $V_{x_j}$ can be joint by a rational curve in the irreducible class~$\beta$. Then, by \cite[Theorem 0.1]{CMSB}, the normalization of a general $V_{x_j}$ is the $r$-dimensional projective space.
\end{comment}

Proposition \ref{prop2.1} shows that $\Mbar_{0,1}(X,\beta)$ is pure of the expected dimension~$2n - 1$ and justifies the decomposition~\eqref{decomp1}. Similar arguments have also appeared in \cite[Theorem 4.4]{AV} and \cite[Proposition 4.10]{BaL}.

\begin{comment}
\begin{rmk}\label{remark2.2}
Applying \cite[Theorem 0.1]{CMSB}, we further deduce that the normalization of a general $V_x$ in \eqref{vx} is isomorphic to the projective space $\BP^r$. 
\end{rmk}
\end{comment}

\begin{comment}
The holomorphic $2$-form also forces $r \leq n$ in the proof of Proposition~\ref{prop2.1}. We thus obtain the following corollary which gives the decomposition \eqref{decomp1}.

\begin{cor}\label{cor2.3}
A general fiber of \eqref{resev} is of dimension at most $n-1$. 
\end{cor}
\end{comment}

\subsection{Gromov--Witten theory}
%We recall the background from Gromov--Witten theory that we need.
%The proofs of Theorems~\ref{mainthm} and \ref{A3} use \emph{Gromov--Witten theory}  We briefly recall the background that we need.

Let $X$ be a holomorphic symplectic variety of dimension~$2n$, and let $\beta \in H_2(X, \mathbb{Z})$ be an arbitrary curve class. By Li--Tian \cite{LT} and Behrend--Fantechi~\cite{BF}, the moduli space of stable maps $\Mbar_{0,m}(X,\beta)$ carries a (reduced\footnote{Since $X$ is holomorphic symplectic, the (standard) virtual fundamental class on the moduli space vanishes. The theory is nontrivial only after reduction; see \cite[Section 2.2]{MP} and~\cite[Section 0.2]{Ob2}. The virtual fundamental class is always assumed to be reduced in this paper.})
virtual fundamental~class
\[ [ \Mbar_{0,m}(X,\beta) ]^{\text{vir}} \in H_{2\mathrm{vdim}}( \Mbar_{0,m}(X,\beta), \BQ). \]
%Gromov--Witten theory is the study of this cycle class. 
It has the following basic properties.

\begin{enumerate}
\item[(a)] \textit{Virtual dimension.} The virtual fundamental class is of dimension
\begin{equation}\label{vdim}
\mathrm{vdim} =  2n - 2 + m.
\end{equation}

\item[(b)] \textit{Expected dimension.} If $\Mbar_{0,m}(X,\beta)$ is pure of the expected dimension~\eqref{vdim}, then the virtual and the ordinary fundamental classes agree:
\[ [ \Mbar_{0,m}(X, \beta) ]^{\text{vir}} = [ \Mbar_{0,m}(X, \beta) ]. \]

\item[(c)] \textit{Deformation invariance.} Let $\pi : \CX \to B$ be a family of holomorphic symplectic varieties, and let $\beta \in H^0(B, R \pi^{4n-2}_{\ast} \BZ)$ be a class
which restricts to a curve class in $H_2(X_b,\BZ)$ on each fiber.\footnote{We have suppressed an application of Poincar\'e duality here.
Same with the definition of $\GW_{\beta}$ and $\Phi_2$ in Section~\ref{gwcorr} below.} Then there exists a class on the moduli space of relative stable maps
\[ [ \Mbar_{0,m}(\CX/B,\beta) ]^{\text{vir}} \in H_{2(\mathrm{vdim} + \dim B)} ( \Mbar_{0,m}(\CX/B,\beta), \BQ) \]
such that for every fiber $ X_b \hookrightarrow \CX$, the inclusion $\iota_b: b \hookrightarrow B$ induces
\[ \iota_b^{!} [ \Mbar_{0,m}(\CX/B,\beta) ]^{\text{vir}} = [ \Mbar_{0,m}(X_b, \beta) ]^{\text{vir}}. \]
Here $\iota_b^{!}$ is the refined Gysin pullback. In particular, intersection numbers of $[ \Mbar_{0,m}(X, \beta) ]^{\text{vir}}$ against cohomology classes pulled back from $X$
via the evaluation maps
\begin{equation*} \ev_i : \Mbar_{0,m}(X,\beta) \to X,\quad (f, x_1, \ldots, x_m) \mapsto f(x_i)
\end{equation*}
are invariant under deformations of $(X,\beta)$ which keep $\beta$ of Hodge type.
\end{enumerate}

\subsection{Gromov--Witten correspondence} \label{gwcorr}
%Our main tool for studying the decomposition \eqref{decomposition} is the Gromov--Witten correspondence \eqref{GWcor}. For our purpose, it is convenient to factorize the Gromov--Witten correspondence as follows. 
%Let $X$ be a holomorphic symplectic variety of dimension $2n$, and let $\beta \in H_2(X,\BZ)$ be a curve class. 
%We introduce certain operators on the cohomology groups $H^i(X,\BQ)$ from Gromov--Witten theory. The operators will be used to detect the components in the decomposition \eqref{decomp1}.

Let $X, \beta$ be as in Section~\ref{Subsection_Dimensions}. The evaluation maps from the $2$-pointed moduli space
\[
\begin{tikzcd}
{} & \Mbar_{0,2}(X,\beta) \ar{rd}{\ev_2} \ar[swap]{ld}{\ev_1} & \\
X & & X
\end{tikzcd}
\]
induce an action on cohomology:
\begin{equation} \label{GWcor}
\GW_{\beta} : H^i(X, \mathbb{Q}) \to H^i(X, \mathbb{Q}),\quad \gamma \mapsto \ev_{2 \ast}( \ev_1^{\ast}\gamma \cap [ \Mbar_{0,2}(X,\beta) ]^{\text{vir}} ).
\end{equation}
We call \eqref{GWcor} the \emph{Gromov--Witten correspondence}.
%\footnote{We have suppressed an application of Poincar\'e duality in the definition of $\mathsf{GW}_\beta$. Same with the definition of $\Phi_2$ below.}

We introduce a factorization of \eqref{GWcor} as follows. Consider the diagram 
\begin{equation} \label{341234}
\begin{tikzcd}
\Mbar_{0,1}(X,\beta) \ar{r}{\ev} \ar{d}{p} & X \\
\Mbar_{0,0}(X,\beta)
\end{tikzcd}
\end{equation}
with $p$ the forgetful map (which is flat). We define morphisms
\begin{gather*}
\Phi_1: H^i(X, \BQ) \rightarrow H_{4n-2-i}(\Mbar_{0,0}(X,\beta), \BQ), \quad \!\gamma \mapsto p_\ast (\mathrm{ev}^\ast \gamma \cap [\Mbar_{0,1}(X,\beta)]^{\mathrm{vir}}), \\
\Phi_2 = \mathrm{ev}_\ast p^\ast: H_{4n-2-i}(\Mbar_{0,0}(X,\beta), \BQ) \rightarrow H^i(X, \BQ).
\end{gather*}
Since $\beta$ is irreducible, there is a Cartesian diagram of forgetful maps
\begin{equation*} 
\begin{tikzcd}
 {} & \Mbar_{0,2}(X,\beta)\arrow{dr}{} \arrow{dl}{} & \\
\Mbar_{0,1}(X,\beta) \arrow{dr}{} & & \Mbar_{0,1}(X,\beta)\arrow{dl}{} \\
{} & \Mbar_{0,0}(X,\beta). &
\end{tikzcd}
\end{equation*}
Hence the Gromov--Witten correspondence \eqref{GWcor} factors as
\begin{equation}\label{factorization}
{\GW}_\beta = \Phi_2 \circ \Phi_1: H^i(X,\BQ) \rightarrow H^i(X, \BQ).
\end{equation}

%In particular, this implies that the Gromov--Witten correspondence
%is equivariant with respect to cup product with the symplectic form $\sigma$ on $X$,
%\[ \GW_{\beta}(\sigma a) = \sigma \GW_{\beta}(a) \]
%for all $a \in H^{\ast}(X)$.

\subsection{Hodge classes}
Now let $(X, H)$ be a very general polarized holomorphic symplectic $4$-fold of $K3^{[2]}$ type. It is shown in \cite[Section 3]{OG} that the Hodge classes in $H^4(X, \BQ)$ are spanned by $H^2$ and $c_2(X)$.

A surface $\Sigma \subset X$ is \emph{Lagrangian} if the holomorphic $2$-form $\sigma$ on $X$ restricts to zero on $\Sigma$. The class of any Lagrangian surface is a positive multiple of
%implies that the class of $S$ lies in a $1$-dimensional subspace of $\langle H^2, c_2(X) \rangle \subset H^4(X, \BQ)$. 
%In fact, by a direct calculation, every Lagrangian surface in $X$ has class proportional to 
\begin{equation}\label{Markman_class}
v_X= 5H^2 - \frac{1}{6}(H,H)c_2(X) \in H^4(X,\BQ),
\end{equation}
where $(-,-)$ is the Beauville--Bogomolov form on $H^2(X, \BZ)$.\footnote{This follows from a direct calculation of the constraint $[\Sigma] \cdot \sigma=0 \in H^6(X, \mathbb{Q})$. The class $v_X$ was first calculated by Markman.}

\begin{prop}\label{prop2.4}
If $(X, H)$ is very general of $K3^{[2]}$ type and $\beta \in H_2(X, \mathbb{Z})$ is the primitive curve class, then for any Hodge class $\alpha \in H^4(X, \BQ)$, the~class
\[{\GW}_\beta(\alpha) \in H^4(X, \BQ)\] is proportional to $v_X$.
\end{prop}

\begin{proof}
We use the factorization \eqref{factorization}. For any Hodge class $\alpha \in H^4(X, \BQ)$, the class
\[
\Phi_1(\alpha) \in H_2(\Mbar_{0,0}(X,\beta), \BQ) 
\]
is represented by curves. Hence ${\GW}_\beta(\alpha)$ can be expressed as a linear combination of classes of the form
\[
[\mathrm{ev}(p^{-1}(C))] \in H^4(X, \BQ)
\]
with $C \subset \Mbar_{0,0}(X,\beta)$ a curve.

Moreover, we have
\[
\mathrm{ev}^\ast \sigma  = p^\ast \sigma' 
\]
for some holomorphic $2$-form $\sigma'$ on $\Mbar_{0,0}(X,\beta)$. Hence any surface of the form $\mathrm{ev}(p^{-1}(C))$ is Lagrangian, and the proposition follows.
\end{proof}

Proposition \ref{prop2.4} implies that the class $v_X$ in \eqref{Markman_class} is an eigenvector of the Gromov--Witten correspondence
 \[
 {\GW}_\beta: H^4(X, \BQ) \rightarrow H^4(X, \BQ).
 \]
An explicit formula for $\GW_{\beta}$ was calculated in \cite{Ob2} and is recalled in Section~\ref{appcor}.

\section{Gromov--Witten calculations} \label{GW_Appendix}

In this Section, we prove Theorem \ref{A3} using formulas for the $1$-pointed Gromov--Witten class in the $K3^{[n]}$ case based on \cite{Ob2}. We also present formulas for the Gromov--Witten correspondence in the $K3^{[2]}$ case, which will be used in Section \ref{sec:1}. 

%In particular, following the notation in Sections 2 and 3, we obtain
%\[
%\ev_{\ast} [ \Mbar_{0,1}(F,\beta) ]^{\textup{vir}} =  60 H \in H^2(F, \BQ)
%\]
%and 
%\[
%\GW_\beta(c) = 945c \in H^4(F, \BQ)
%\]
%for the Fano varieties of lines of cubic 4-folds.

\subsection{Quasi-Jacobi forms}
Jacobi forms are holomorphic functions in variables\footnote{Let $\BH = \{ \tau \in \BC : \mathrm{Im}(\tau)>0 \}$ denote the upper half-plane.} $(\tau,z) \in \BH \times \BC$ with modular properties; see \cite{EZ} for an introduction.
Here we will consider Jacobi forms as formal power series in the variables
\[ q = e^{2 \pi i \tau}, \quad y = -e^{2 \pi i z} \]
expanded in the region $|q|<|y|<1$. 

Recall the Jacobi theta function
\[ \Theta(q,y) = (y^{1/2} + y^{-1/2}) \prod_{m \geq 1} \frac{ (1 + yq^m) (1 + y^{-1}q^m)}{ (1-q^m)^2 } \]
and the Weierstra{\ss} elliptic function
\[ \wp(q,y) = \frac{1}{12} - \frac{y}{(1+y)^2} + \sum_{m \geq 1} \sum_{d|m} d ((-y)^d - 2 + (-y)^{-d}) q^{m}. \]
Define Jacobi forms $\phi_{k,1}$ of \emph{weight} $k$ and \emph{index} $1$ by
\[ \phi_{-2,1}(q,y) = \Theta(q,y)^2, \quad \phi_{0,1}(q,y) = 12 \Theta(q,y)^2 \wp(q,y). \]

We also require the weight $k$ and index $0$ Eisenstein series
%\[ C_{k}(q) = -\frac{B_k}{k \cdot k!} + \frac{2}{k!} \sum_{n \geq 1} \sum_{d|n} d^{k-1} q^n, \quad k=2,4,6 \]
\[ E_{k}(q) = 1 - \frac{2k}{B_{k}} \sum_{m \geq 1} \sum_{d |m} d^{k-1} q^m, \quad k=2,4,6, \]
where the $B_k$ are the Bernoulli numbers, and the modular discriminant
\[ \Delta(q) = \frac{E_4^3 - E_6^2}{1728} = q \prod_{m\geq1} (1-q^m)^{24}. \]

We define the ring of quasi-Jacobi forms of even weight as the free polynomial algebra
\[ \CJ = \BQ[ E_2, E_4, E_6, \phi_{-2,1}, \phi_{0,1}]. \]
The weight/index assignments to the generators induce a bigrading
\[ \CJ = \bigoplus_{k \in \BZ} \bigoplus_{m \geq 0} \CJ_{k,m} \]
by weight $k$ and index $m$.

\begin{lemma}[{\cite[Theorem 2.2]{EZ}}] \label{Lemma_Jac_coeff}
Let $\phi \in \CJ_{\ast, m}$ be a quasi-Jacobi form of index $m \geq 1$.
For all $d, r \in \BZ$, the coefficient $[ \phi ]_{q^d y^r}$ only depends on $2d-\frac{r^2}{2m}$
and the set $\{ \pm [r] \}$, where $[r] \in \BZ/2m \BZ$ is the residue of $r$.
\end{lemma}

By Lemma~\ref{Lemma_Jac_coeff}, we may denote the $q^d y^r$-coefficient of $\phi$ by
\begin{equation} \phi\left[ 2d-\frac{r^2}{2m},\, \pm [r] \right] = [ \phi ]_{q^d y^r}. \label{coff_def} \end{equation}
% If the quasi-Jacobi form $\phi$ is of index $1$, then the residue $[r]$ is uniquely determined by $2n- \frac{r^2}{2}$.\footnote{
% If $D = 4n-r^2$, then $[r]=0$ if $D \equiv 0 (4)$ and $[r]=1$ if $D \equiv 3 (4)$}
% In this case we also write
% \[ \phi\left[ 2n- \frac{r^2}{2} \right] := \left[ \phi \right]_{q^n y^r}. \]
If $\phi$ is of index $0$, we set
$\phi[ 2d, 0 ] = [ \phi ]_{q^d}$.
Lemma~\ref{Lemma_Jac_coeff} remains valid if we replace $\phi$ by $f(q) \phi$ for any Laurent series $f(q)$, and we keep the notation as in \eqref{coff_def} for the coefficients.

We will mainly focus on the quasi-Jacobi form
\begin{equation} \label{phi}
\phi = \left( - \wp + \frac{1}{12} E_2 \right) \Theta^2.
\end{equation}
The following are some positivity results.

\begin{lemma} \label{Positivity_1} Let $\phi$ be as in \eqref{phi}. Then
$\phi[ D ] \geq 0$ for all $D$ and
\[ \phi[ D ] > 0 \Longleftrightarrow D = 2n- \frac{r^2}{2} \geq 0 \text{ for some } n, r \in \BZ. \]
\end{lemma}
\begin{proof}
By the Jacobi triple product, we have $\Theta = \vartheta_1 / \eta^3$ where
\[ \vartheta_1(q,y) = \sum_{n \in \BZ + \frac{1}{2}} y^n q^{\frac{1}{2} n^2}, \quad \eta(q) = q^{\frac{1}{24}} \prod_{n \geq 1} (1-q^n). \]
If we write $\Theta = \sum_{n,r} c(n,r) q^n y^r$, we therefore get
\[ c(n,r) > 0 \Longleftrightarrow \Big( r \in \frac{1}{2} \BZ \setminus \BZ \text{ and } 2n \geq r^2 - \frac{1}{4} \Big) \]
and $c(n,r) = 0$ otherwise. 
By the explicit expressions for the action of differential operators on quasi-Jacobi forms in \cite[Appendix B]{Ob2},
we have the identity
\[ \phi = \Theta^2 D_y^2 \log \Theta = D_y^2(\Theta) \Theta - D_y(\Theta)^2. \]
Hence
\begin{equation}
\begin{aligned} \big[ \phi \big]_{q^n y^k}
& = \sum_{\substack{n = n_1 + n_2 \\ k = k_1 + k_2}} c(n_1, k_1) c(n_2, k_2) (k_1^2 - k_1 k_2) \\
& = \frac{1}{2} \sum_{\substack{n = n_1 + n_2 \\ k = k_1 + k_2}} c(n_1, k_1) c(n_2, k_2) (k_1 - k_2)^2 \geq 0.
\label{134}
\end{aligned}
\end{equation}
Since $\phi$ is quasi-Jacobi, the coefficient $\big[ \phi \big]_{q^n y^k}$ only depends on
$4n-k^2$, hence we may assume $k \in \{ 0, 1 \}$. The result now follows from \eqref{134} by a direct check.
\end{proof}

\begin{comment}
\begin{lemma} \label{Positivity_2}
\[ \left( \frac{-E_2 \Theta^2}{\Delta} \right) \Big[ D \Big] > 0 \ \Longleftrightarrow \ D = 2n- \frac{r^2}{2} > 0 \text{ for some } n, r \in \BZ. \]
\end{lemma}
\begin{proof} Using $\Delta = \eta^{24}$, $\Theta = \vartheta_{1}/\eta^3$ and $D_q \eta = \frac{1}{24} E_2 \eta$
we have
\[ \frac{-E_2 \Theta^2}{\Delta} 
%= -E_2 \frac{\vartheta_1^2}{\eta^{30}} 
= \frac{24}{30} D_q\left( \frac{1}{\eta^{30}} \right) \vartheta_1^2. \]
The claim follows by inspection of the expansions
\[ \frac{24}{30} D_q\left( \frac{1}{\eta^{30}} \right) = - q^{-\frac{5}{4}} - 6 q^{-\frac{1}{4}} + \sum_{n \geq 1} a(n) q^{n - \frac{1}{4}} \]
where $a(n) > 7$ for all $n \geq 1$, and %(since we only need to consider coefficients $y^0, y^1$)
\begin{align*}
\left[ \vartheta_1^2 \right]_{y^0} & = \sum_{n \in \BZ + \frac{1}{2}} q^{n^2} = 2 q^{\frac{1}{4}}( 1 + q^2 + q^6 + \ldots ) \\
\left[ \vartheta_1^2 \right]_{y^1} & = q^{\frac{1}{4}} \sum_{n \in \BZ} q^{n^2} = q^{\frac{1}{4}} (1 + 2 q + 2 q^4 + \ldots) \qedhere
\end{align*}
% \[
% \left[ \vartheta_1^2 \right]_{y^0} = \sum_{n \in \BZ + \frac{1}{2}} q^{n^2} = 2 q^{\frac{1}{4}}( 1 + q^2 + \ldots ),
% \ \left[ \vartheta_1^2 \right]_{y^1} = q^{\frac{1}{4}} \sum_{n \in \BZ} q^{n^2} = 2 q^{\frac{1}{4}} (1 + q^4 + \ldots) \qedhere
% \]
\end{proof}
\end{comment}

\subsection{Beauville--Bogomolov form}\label{BB_Section}
Let $X$ be a holomorphic symplectic variety of dimension $2n$.
The Beauville--Bogomolov form on $H^2(X,\BZ)$ induces an embedding
\[ H^2(X,\BZ) \hookrightarrow H_2(X,\BZ), \quad \alpha \mapsto (\alpha, - ), \]
which is an isomorphism after tensoring with $\BQ$. Let
\begin{equation} \label{extbb}
( - , - ) : H_2(X,\BZ) \times H_2(X,\BZ) \to \BQ
\end{equation}
denote the unique $\BQ$-valued extension of the Beauville--Bogomolov form.

If $X$ is of $K3^{[n]}$ type with $n \geq 2$, there is an isomorphism of abelian~groups
\[ r : H_2(X,\BZ)/H^2(X,\BZ) \to \BZ / (2n-2) \BZ \]
such that $r(\alpha) = 1$ for some $\alpha \in H_2(X, \BZ)$ with $(\alpha, \alpha) = \frac{1}{2-2n}$.
The morphism~$r$ is unique up to multiplication by $\pm 1$. 

\subsection{Curve classes} \label{Subsection_curve_classes}
Consider a pair $(X,\beta)$ where $X$ is a holomorphic symplectic variety of $K3^{[n]}$ type, and $\beta \in H_2(X,\BZ)$ is a primitive curve class. The curve class $\beta$ has the following invariants:
\begin{enumerate}
\item[(i)] the Beauville--Bogomolov norm $(\beta, \beta) \in \BQ$, and
\item[(ii)] the residue $[ \beta ] \in H_2(X,\BZ) / H^2(X,\BZ)$.
\end{enumerate}
The \emph{residue set} of $\beta$ is the subset
\begin{equation*} \label{res_set} \pm [\beta] = \{ \pm r([\beta]) \} \subset \BZ / (2n-2) \BZ. \end{equation*}
It is independent of the choice of map $r$. If $n=1$, we set $\pm [\beta]=0$.

Given a (quasi-)Jacobi form $\phi$ of index $m=n-1$, we define
\[ \phi_{\beta} = \phi[ (\beta, \beta) , \pm [ \beta ] ]. \]

% The heat operator \eqref{Heat_operator} therefore satisfies
% \begin{equation} \CH(\phi)_{\beta} = (\beta, \beta) \phi_{\beta}. \label{Heat_coff_relation}\end{equation}

By Markman \cite{Mar} (see also \cite[Lemma 23]{Ob1}), two pairs $(X,\beta)$ and $(X', \beta')$ are deformation equivalent through a family of holomorphic symplectic manifolds
which keeps the curve class of Hodge type if and only if the norms and the residue sets of $\beta$ and $\beta'$ agree.
Hence, by identifying $H^{\ast}(X)$ with $H^{\ast}(X')$ via parallel transport and by property (c) of the virtual fundamental class,
the Gromov--Witten invariants of the pairs $(X,\beta)$ and $(X',\beta')$ are equal.\footnote{The (reduced) virtual fundamental class can also be defined via symplectic geometry and the twistor space of $X$; see \cite{BL}. Hence, the Gromov--Witten invariants are invariant also under (nonnecessarily algebraic) symplectic deformations of $(X, \beta)$ which keep~$\beta$ of Hodge type.
The invariance under nonalgebraic deformations is not needed for our application
to the Fano variety of lines in a cubic $4$-fold.} 

\subsection{Proof of Theorem \ref{A3}} \label{appuni}
Recall from \eqref{phi} the quasi-Jacobi form $\phi$. % introduced in~\eqref{phi}.

\begin{thm}[\cite{Ob2}] \label{thm_Uniruled}
Let $X$ be a holomorphic symplectic variety of $K3^{[n]}$ type, and let $\beta \in H_2(X,\BZ)$ be a primitive curve class. Then we have
\[ \ev_{\ast} [ \Mbar_{0,1}(X,\beta) ]^{\textup{vir}} 
= 
\left( \frac{\phi^{n-1}}{\Delta} \right)_{\beta} h \in H^2(X,\BQ)\]
where $h = (\beta, -) \in H^2(X,\BQ)$ is the dual of $\beta$ with respect to \eqref{extbb}.
\end{thm}

For the readers' convenience, we provide a proof of Theorem~\ref{thm_Uniruled} at the end of this section. Theorem~\ref{thm_Uniruled} together with the positivity of the Fourier coefficients of $\phi$ implies Theorem \ref{A3}.

\begin{proof}[Proof of Theorem \ref{A3}]
By Lemma~\ref{Positivity_1} the criterion in Theorem~\ref{A3} holds
if and only if 
\[ \left( \frac{\phi^{n-1}}{\Delta} \right)_{\beta} > 0, \]
%$( \phi^{n-1} / \Delta )_{\beta} > 0$
hence by Theorem~\ref{thm_Uniruled} if and only if the pushforward
$\ev_{\ast} [ \Mbar_{0,1}(X,\beta) ]^{\textup{vir}}$ % \in H^2(X,\BQ)$
is nontrivial.
Since the pushforward is a class in $H^2(X,\BQ)$ supported on a uniruled subvariety, the first claim follows.
The second claim follows from Proposition~\ref{prop2.1} and property (b) of the virtual fundamental class.
\end{proof}

%As a consequence, a very general holomorphic symplectic varieties of $K3^{[n]}$ type may contain \emph{no} uniruled divisor swept out by rational curves in the primitive curve class. The first instance is given by a very general pair $(X,\beta)$ of $K3^{[8]}$ type with
%$(\beta,\beta) = \frac{3}{14}$ and $\pm [\beta] = \pm [5]$.\footnote{Such a pair $(X,\beta)$ can be obtained by deforming $(\Hilb^8(\mathsf{S}), \beta')$, where $\mathsf{S}$ is a $K3$ surface of genus $2$ with polarization $\mathsf{H}$ and $\beta' = \mathsf{H} + 5 \mathsf{A}$
%with $\mathsf{A}$ the exceptional curve class.}
In the $K3^{[2]}$ case, we define
\[f = \frac{\phi}{\Delta} = \left( - \wp + \frac{1}{12} E_2 \right) \frac{\Theta^2}{\Delta}.\]
The first few values of $f_\beta$ are listed in the following table.\footnote{When $n = 2$, the value $(\beta, \beta) \in \BQ$ uniquely determines $\pm[\beta] \subset \BZ/2\BZ$.}

\begin{table}[ht]
{\renewcommand{\arraystretch}{1.5}\begin{tabular}{| c | c | c | c | c | c | c | c | c | c | c |}
\hline
$\!(\beta,\beta)$\! & $-\frac{5}{2}$ & $-2$ & $- \frac{1}{2}$ & $0$ & $\frac{3}{2}$ & $2$ & $\frac{7}{2}$ & $4$ & $\frac{11}{2}$ & $6$ \\
\hline
$f_{\beta}$ & $0$ & $1$ & $4$ & $30$ & $120$ & $504$ & $1980$ & $6160$ & $23576$ & $60720$ \\
\hline
\end{tabular}}
\caption{The first few multiplicities of uniruled divisors for $K3^{[2]}$.} \label{uniruled_table}
\vspace{-15pt}
\end{table}
% 
% -2 1
% 0 30
% 2 504
% 4 6160
% 6 60720
% 8 509796
% 10 3770400
% 12 25128864
% 14 153406575
% 16 868419160
% 18 4602252312
% 20 23007595680
% 22 109175251472
% -5/2 0
% -1/2 4
% 3/2 120
% 7/2 1980
% 11/2 23576
% 15/2 225540
% 19/2 1835064
% 23/2 13150200
% 27/2 84968400
% 31/2 503381772
% 35/2 2768699120
% 39/2 14275014768
% 43/2 69520941000

\subsection{Gromov--Witten correspondence} \label{appcor}
In this section, we specialize to the $K3^{[2]}$ case. Recall the Gromov--Witten correspondence $\GW_{\beta}$ in \eqref{GWcor}. We also define
\[ g = \left( - \frac{12}{5} \wp -E_2 \right) \frac{\Theta^2}{\Delta}. \]

\begin{thm}[\cite{Ob2}] \label{ThmGWb_1} 
Let $X$ be a holomorphic symplectic $4$-fold of $K3^{[2]}$ type, and let $\beta \in H_2(X,\BZ)$ be a primitive curve class. If $(\beta, \beta) \neq 0$, then~$\GW_{\beta}$ is diagonalizable with eigenvalues
\[ \lambda_0 = 0, \quad \lambda_1 = (\beta, \beta) f_{\beta}, \quad \lambda_2 = (\beta, \beta) g_{\beta}, \]
and eigenspaces
\[ V_{\lambda_1} = \BQ\langle  h, h^3,  (h e_i)_{i=1,\ldots,22}\rangle, \quad V_{\lambda_2} = \BQ v. \]
Here $h = (\beta, -) \in H^2(X,\BQ)$ is the dual of $\beta$ with respect to \eqref{extbb}, $\{ e_i \}_{i=1,\ldots,22}$ is a basis of the orthogonal of $h$ in $H^2(X,\BQ)$,
and
\[v = 5 h^2 - \frac{1}{6}(\beta,\beta) c_2(X) \in H^4(X,\BQ).\]
\end{thm}

\begin{comment}
\vspace{5pt}
By Theorem~\ref{ThmGWb_1}, if $(\beta, \beta) \neq 0$, then the matrix of $\GW_{\beta}$ in an eigenbasis reads %for $(\beta, \beta) \neq 0$ is similar to the following matrix:
\[
\GW_{\beta}
\sim
\left(
\begin{array}{c|cc|ccc|cc|c}
0 & & & & & & & & \\
\hline
& \lambda_2 & & & & & & & \\
& & (0)^{\oplus 22} & & & & & & \\
\hline
& & & \lambda_1 & & & & &\\
& & & & \left( \lambda_2 \right)^{\oplus 22} & & & & \\
& & & & & (0)^{\oplus 253} & & & \\
\hline
& & & & & & \lambda_2 & & \\
& & & & & & & (0)^{\oplus 22} & \\
\hline
& & & & & & & & 0
\end{array}
\right)
\]
where the dividing lines correspond to the decomposition of $H^{\ast}(X)$ by cohomological degree.
\end{comment}

One can show that the eigenvalues $\lambda_1, \lambda_2$ are integral, and if $(\beta,\beta)>0$ then $\lambda_2 > \lambda_1 > 0$.
The first few eigenvalues are listed in Table~\ref{eigenvalue_table}.

\begin{table}[ht]
{\renewcommand{\arraystretch}{1.5}\begin{tabular}{| c | c | c | c | c | c | c | c | c | c | c |}
\hline
\!$(\beta,\beta)$\! & $-\frac{5}{2}$ & $-2$ & $- \frac{1}{2}$ & $0$ & $\frac{3}{2}$ & $2$ & $\frac{7}{2}$ & $4$ & $\frac{11}{2}$ & $6$ \\
%\hline
%$\pm[\beta]$ \\
\hline
$\lambda_1$ & $0$ & $-2$ & $-2$ & $0$ & $180$ & $1008$ & $6930$ & $24640$ & $129668$ & $364320$ \\
\hline
$\lambda_2$ & $3$ & $0$ & $0$ & $0$ & $945$ & $3840$ & $53760$ & $ 138240$ & $1237005$ & $2661120$ \\
\hline
\end{tabular}}
\caption{The first eigenvalues of $\GW_{\beta}$ for $K3^{[2]}$.}
\label{eigenvalue_table}
\vspace{-15pt}
\end{table}

\begin{comment}
\begin{thm}[\cite{Ob2}] \label{ThmGWb_2} Let $(X,\beta)$ be as in Theorem~\ref{ThmGWb_1}. If $(\beta, \beta) = 0$ then
$\GW_{\beta}^3=0$ and
\begin{align*}
\mathrm{Im}(\GW_{\beta}) & = \mathrm{Span}(  h, v, \beta, (h e_i)_{i=1,\ldots,22}) \\
\mathrm{Im}(\GW_{\beta}^2) & = \mathrm{Span}( v ).
%\mathrm{Im}(\GW_{\beta}^3) & = \mathrm{Span}( h, h e_i, i=1,\ldots, 22, v, \beta ) \\
\end{align*} 
\end{thm}
\subsection{Proof of Theorem~\ref{ThmGWb_2}} 
This follows from \eqref{asdsda} with $a_{\beta} = 15, b_{\beta} = -6/5, c_{\beta}=0, d_{\beta}=30$ if $(\beta,\beta)=0$. \qed
\end{comment}

\subsection{Proof of Theorem~\ref{thm_Uniruled}}
A very general pair $(X, \beta)$ has Picard rank~$1$.\footnote{In this statement, we allow $X$ to be a holomorphic symplectic manifold.} %with curve class $\beta$ every divisor is a multiple of $\beta^{\vee}$.
Hence there exists $N_{\beta} \in \BQ$ such that
\[ \ev_{\ast} [ \Mbar_{0,1}(X,\beta) ]^{\text{vir}} = N_{\beta} h \in H^2(X, \mathbb{Q}). \]
By specialization, this also holds for any pair $(X, \beta)$ as in Theorem \ref{thm_Uniruled}.

We will evaluate $N_{\beta}$ on the Hilbert scheme of $n$ points on an elliptic $K3$ surface $\mathsf{S}$ with a section.
By Section~\ref{Subsection_curve_classes}, we may assume
\begin{equation} \beta = \mathsf{B} + (d+1)\mathsf{F} + r\mathsf{A} \in H_2(\Hilb^n(\mathsf{S}), \mathbb{Z}), \quad d \geq -1, \ r \in \BZ, \label{betaclass} \end{equation}
where $\mathsf{B}, \mathsf{F} \in H_2(\mathsf{S},\BZ)$ are the classes of the section and fiber of the elliptic fibration,
and $\mathsf{A} \in H_2(\Hilb^n(\mathsf{S}), \mathbb{Z})$ is the class of an exceptional curve (for~$n \geq 2$). Here we apply the natural identification 
\[H_2(\Hilb^n(\mathsf{S}),\BZ) \simeq H_2(\mathsf{S},\BZ) \oplus \BZ \mathsf{A}.\]

Let $\mathsf{F}_0 \subset \mathsf{S}$ be a nonsingular fiber, and let $x_{1}, \ldots, x_{n-1} \in \mathsf{S} \setminus \mathsf{F}_0$ be distinct points. Consider the curve
\[ \mathsf{C} = \{ x_1 + \cdots + x_{n-1} + x' : x' \in \mathsf{F}_0 \} \subset \Hilb^n(\mathsf{S}). \]
Then $\int_{[\mathsf{C}]} h = 1$ and hence by the first equation in \cite[Theorem 2]{Ob2}, we~find
\[
N_{\beta}
= \int_{[ \Mbar_{0,1}(X,\beta) ]^{\text{vir}}} \ev^{\ast} [\mathsf{C}]
= \left[ \frac{\phi^{n-1}}{\Delta} \right]_{q^d y^r}
%= g\left[ (2n-2) \left( 2d + \frac{k^2}{2 - 2n} \right), [k] \right]
%= \left( \frac{\phi^{n-1}}{\Delta} \right)\left[ 2d - \frac{k^2}{2 (n-1)} , \pm [k] \right]
= \left( \frac{\phi^{n-1}}{\Delta} \right)_{\beta}. \tag*{\qed}
\]

\subsection{Proof of Theorem~\ref{ThmGWb_1}}
Consider the $2$-pointed class
\begin{align*}
\Z_{\beta} & = \ev_{12\ast} [ \Mbar_{0,2}(X,\beta) ]^{\text{vir}} \in H^8(X \times X, \mathbb{Q}).
\end{align*}
By the divisor equation \cite{FP} and Theorem~\ref{thm_Uniruled}, we have
\begin{equation*}
\int_{\Z_{\beta}} \gamma \otimes \delta
%= \int_{\beta} \delta \int_{[\Mbar_{0,1}(X,\beta)]^{\textup{vir}}} \ev_1^{\ast}(\gamma)
=  \left(\int_\beta \delta \int_\gamma h\right) f_{\beta}
\end{equation*}
for all $\delta \in H^2(X, \BQ)$ and $\gamma \in H^6(X, \BQ)$.\footnote{We have suppressed an application of Poincar\'e duality here.} Hence
\begin{gather*}
\GW_{\beta}(\delta) = \left(\int_\beta \delta\right) f_{\beta} h \in H^2(X, \BQ),\\
\GW_{\beta}(\gamma) = \left(\int_\gamma h \right) f_{\beta} \beta \in H^6(X, \BQ).
\end{gather*}

Now consider the $(4,4)$-K\"unneth factor of $\Z_{\beta}$,
\[ \Z_{\beta}^{4,4} \in H^{4}(X) \otimes H^4(X). \]
By monodromy invariance under the group $\mathrm{SO}( H^2(X,\BC), h)$, we~have
\begin{multline*} \label{gcbcv}
\Z_{\beta}^{4,4} = a_{\beta} h^2 \otimes h^2 + b_{\beta} ( h^2 \otimes c_2(X) + c_2(X) \otimes h^2 ) + c_{\beta} c_2(X) \otimes c_2(X) \\ + d_{\beta} (h \otimes h) c_{BB} + e_{\beta} [ \Delta_{X} ]^{4,4}
\end{multline*}
for some $a_{\beta}, b_{\beta}, c_{\beta}, d_{\beta}, e_{\beta} \in \BQ$; see \cite[Section 4]{HHT}. Here
\[ c_{BB} \in \Sym^2(H^2(X, \BQ)) \subset H^2(X, \BQ) \otimes H^2(X, \BQ) \]
is the inverse of the Beauville--Bogomolov class.
%and $c_{BB} \in H^4(X, \BQ)$ is its image under cup product.

Since $\int_{\Z_{\beta}} \sigma^2 \otimes \bar{\sigma}^2 = 0$, we have $e_{\beta}=0$.
Also, since the Gromov--Witten correspondence is equivariant with respect to multiplication by $\sigma$, we find
\[ \GW_{\beta}(h \sigma) = \GW_{\beta}(h)\sigma = (\beta,\beta) f_{\beta} h \sigma. \]
Hence $d_{\beta} = f_{\beta}$. Together with Proposition~\ref{prop2.4} 
and $\int_X v^2 = 48 (\beta,\beta)^2 \neq 0$, this implies
\begin{equation}  \Z_{\beta}^{4,4} = \psi_{\beta} \frac{v \otimes v}{48 (\beta,\beta)^2} + f_{\beta} (h \otimes h) \left( c_{BB} - \frac{h \otimes h}{(\beta,\beta)} \right) \label{dfsd} \end{equation}
for some $\psi_{\beta} \in \BQ$. It remains to determine $\psi_{\beta}$.

As in the proof of Theorem~\ref{thm_Uniruled},
let $\mathsf{S}$ be an elliptic $K3$ surface with a section,
and let $\beta$ be as in \eqref{betaclass}. Consider the fiber class of the Lagrangian fibration $\Hilb^2(\mathsf{S}) \to \p^2$ induced by the elliptic fibration~$\mathsf{S} \to \BP^1$,
\[ \mathsf{L} \in H^4(\Hilb^2(\mathsf{S}), \BQ). \]
We have
\[ \int_{\Hilb^2(\mathsf{S})} h^2\mathsf{L} = 2, \quad \int_{\Hilb^2(\mathsf{S})} v\mathsf{L} =10, \quad \int_{\Hilb^2(\mathsf{S}) \times \Hilb^2(\mathsf{S})} ( h\mathsf{L} \otimes h\mathsf{L} ) c_{BB} = 0. \]
Then \cite[Theorem 1]{Ob2} and \eqref{dfsd} imply the relation
\[
\left( \frac{\Theta^2}{\Delta} \right)_{\beta} = \int_{\Z_{\beta}} \mathsf{L} \otimes \mathsf{L} = \frac{10^2}{48 (\beta,\beta)^2} \psi_{\beta} - \frac{2^2}{(\beta,\beta)} f_{\beta}.
\]
Hence
\begin{equation*} 
\psi_{\beta} 
%= \frac{12 (\beta,\beta)}{25} \left( 4 f_{\beta} + (\beta,\beta) \left( \frac{\Theta^2}{\Delta} \right)_{\beta} \right) 
= \frac{12(\beta,\beta)}{25} \left( 4 f + \CH_1\left( \frac{\Theta^2}{\Delta} \right) \right)_{\beta}
\end{equation*}
where
\begin{equation*} \CH_m = 2 q \frac{d}{dq} - \frac{1}{2m} \left(y \frac{d}{dy}\right)^2, \quad m \geq 1 \end{equation*}
is the \emph{heat operator}.
Explicit formulas for the derivatives of Jacobi forms can be found in \cite[Appendix B]{Ob2}, and this yields $\psi_{\beta} = (\beta,\beta) g_{\beta}$ as desired. %the theorem by a direct calculation.
\qed

\section{Rational curves in the Fano varieties of lines} \label{sec:1}
We give the proof of Theorem \ref{mainthm}. From now on, let $F$ be the Fano variety of lines in a very general cubic $4$-fold $Y$, and let $\beta \in H_2(F, \mathbb{Z})$ be the primitive curve class.

\subsection{Degeneracy locus}
The variety $F$ is naturally embedded in the Grassmannian $\mathrm{Gr}(2,6)$. Let $\CU$ and $\CQ$ be the tautological bundles of ranks $2$ and~$4$ with the short exact sequence
\[
0 \to \CU \to \BC^6 \otimes \CO_{\mathrm{Gr}(2,6)} \to \CQ \to 0.
\]
We use $\CU_F, \CQ_F$ to denote the restriction of $\CU, \CQ$ on $F$. Let $H = c_1(\CU_F^*)$ be the hyperplane class on $F$ with respect to the Pl\"ucker embedding. By \cite{BD}, the primitive curve class $\beta \in H_2(F, \mathbb{Z})$ is characterized by $\int_\beta H =3$. 

The indeterminacy locus $S$ of the rational map \eqref{rationalmap} consists of lines $l \subset Y$ with normal bundle
\[
\CN_{l/Y} = \CO_l(-1) \oplus \CO_l(1)^{\oplus 2}.
\]
For every line $l \subset Y$ corresponding to $s \in S$, there is a pencil of planes tangent to $Y$ along $l$. The residual lines of this pencil form the rational curve $\phi(p^{-1}(s)) \subset F$. By \cite[Proposition 6]{A}, we have
\[\int_{[\phi(p^{-1}(s))]} H = 3.\]
Hence the curve $\phi(p^{-1}(s))$ lies in the primitive curve class $\beta$. Moreover, by the calculations in \cite[Theorem 8]{A}, we find
\begin{equation} \label{60H}
\phi_*[D] = 60 H \in H^2(F, \mathbb{Q}).
\end{equation}

In \cite{A}, the surface $S$ is shown to be nonsingular, and is expressed as the (rank $\leq 2$) degeneracy locus of the (sheafified) Gauss map
\[
g: \mathrm{Sym}^2(\CU_F) \rightarrow \CQ_F^\ast
\]
associated to the cubic $Y$. Let $\pi: \BP\mathrm{Sym}^2(\CU_F) \rightarrow F$ be the $\BP^2$-bundle and let~$h$ be the relative hyperplane class. Then $S$ is isomorphic to the zero locus~$S'$ of a section of the rank $4$ vector bundle $\pi^\ast \CQ_F^\ast \otimes \CO(h)$ on~$\BP\mathrm{Sym}^2(\CU_F)$. Let $H_{S'}, h_{S'}$ be the restrictions of the divisor classes $\pi^\ast H, h$ on $S'$. 
%There is the following calculation of intersection numbers.

\begin{lem}\label{calculation}
We have 
\[
\int_{S'} H_{S'}^2 = \int_{S'} H_{S'} h_{S'} = \int_{S'} h_{S'}^2 = 315.
\]
\end{lem}

\begin{proof}
Let $c = c_2(\CU_F^*) \in H^4(F, \BQ)$. Since $S' \subset \BP\mathrm{Sym}^2(\CU_F)$ is the zero locus of a section of the vector bundle $\pi^\ast \CQ_F^\ast \otimes \CO(h)$, a direct calculation yields 
\begin{multline*}
[S'] = c_4(\CQ_F^\ast \otimes \CO(h)) = 5(\pi^\ast H^2- \pi^\ast c) h^2 - \frac{35}{6}\pi^\ast H^3\cdot h + \frac{10}{3}\pi^\ast H^4 \\
\in H^8(\BP\mathrm{Sym}^2(\CU_F), \BQ).
\end{multline*}
The lemma follows from the projection formula, the intersection numbers calculated in \cite[Lemma 4]{A}, and the projective bundle formula associated to~$\pi: \BP\mathrm{Sym}^2(\CU_F) \rightarrow F$,
\[
h^3= 3\pi^\ast H \cdot h^2 - (2 \pi^\ast H^2 + 4 \pi^\ast c) h + \frac{5}{3}\pi^\ast H^3 \in H^6(\BP\mathrm{Sym}^2(\CU_F), \BQ). \qedhere
\]
\end{proof}

From the lemma we can deduce the following.

\begin{lemma} \label{H=h}
We have $H_{S'} = h_{S'} \in H^2(S', \BQ)$. \end{lemma}
\begin{proof}
%If $S'$ is connected, the claim follows directly from the Hodge index theorem and Lemma~\ref{calculation}.
%For the general case, l
Let $U \subset \p(H^0(\p^5, \CO(3)))$ be the open locus parametrizing nonsingular cubic $4$-folds and consider the incidence correspondence
%\[ I = \{ (f,\ell) : V(f) \text{ smooth }, \ell \subset Y \text{ of second type }\} \subset \p(H^0(\p^5, \CO(3))) \times \mathrm{Gr}(2,6) \]
%\[ I = \{ (f,\ell) : V(f) \text{ smooth and } \ell \subset V(f) \text{ of second type } \}
\[ I = \{ (Y,\ell) : \ell \subset Y \text{ corresponding to } s \in S \} \subset U \times \mathrm{Gr}(2,6). \]
%of pairs $(Y,\ell)$ where $\ell \subset Y$ is a line of second type.
By \cite[Proof of Lemma 1]{A} and general properties of determinantal varieties,
the fibers of the projection $I \to \mathrm{Gr}(2,6)$ are irreducible.
Using the homogeneity of $\mathrm{Gr}(2,6)$ we find that $I$ is irreducible.
In particular, the monodromy of the projection $I \to U$ acts transitively on the set of connected components of a very general fiber.
Since the restriction $\pi|_{S'} : S' \to S$ is an isomorphism, the same applies to the monodromy of $S'$.
%when varying the cubic fourfold in $U$.

Let $S_i$ be the connected components of $S'$ and write
\[ H_{S'} - h_{S'} = \sum_i a_i \]
with $a_i \in H^2(S_i, \BQ)$. The classes $H_{S'}$ and $h_{S'}$ are monodromy invariant since they are restricted from $\mathrm{Gr}(2,6)$ and~$\p \Sym^2( \CU_F )$. Hence for every monodromy operator $g: H^{\ast}(S',\BQ) \to H^{\ast}(S',\BQ)$
that sends the $i$-th to the $j$-th component, we find $g a_i = a_j$ and
\[ \int_{S'} H_{S'} \cdot a_i = \int_{S'} g H_{S'} \cdot g a_i = \int_{S'} H_{S'} \cdot a_j. \]
In particular, the intersection number $\int_{S'} H_{S'} \cdot a_i$ is independent of $i$.

By Lemma~\ref{calculation} we have $\int_{S'} H_{S'}(H_{S'} - h_{S'}) = 0$, which together with the above implies that
$\int_{S'} H_{S'} \cdot a_i  = 0$ for all $i$. Now $\int_{S'}(H_{S'} - h_{S'})^2 = 0$ and the Hodge index theorem force $a_i = 0$.
\end{proof}

\begin{cor}\label{cor1.3}
If $H_S$ is the restriction of $H$ to $S$, then we have
\[
c_1(S) = -3H_S \in H^2(S, \BQ).
\]
\end{cor}
\begin{proof}
The surface $S'$ is the zero locus of a regular section of the vector bundle $\pi^{\ast} \CQ_{F}^{\ast} \otimes \CO(h)$.
Hence by the adjunction formula and Lemma~\ref{H=h}, the first Chern class $c_1(S)$ is proportional to $H_S$.
The coefficient is determined by a calculation of intersection numbers; see \cite[Remark in Section 2]{A}.
\end{proof}

\subsection{Divisorial contribution}
By Proposition \ref{prop2.1}, the moduli space of stable maps $\Mbar_{0, 1}(F, \beta)$ is pure of dimension $3$. Recall the decomposition~\eqref{decomp1},
\[\Mbar_{0, 1}(F, \beta) = M^0 \cup M^1,\]
such that a general fiber of $\mathrm{ev}: M^i \to \mathrm{ev}(M^i) \subset F$ is of dimension $i$. We first analyze the component $M^0$.

By construction, the family of maps $p : D \to S$ in \eqref{uniruled} has a factorization
\[\phi: D \to M^0 \xrightarrow{\mathrm{ev}} F.\]
We have seen in \eqref{60H} that
\[\phi_*[D] = 60 H \in H^2(F, \mathbb{Q}).\]
On the other hand, by Theorem~\ref{thm_Uniruled}\footnote{By \cite{BD}, we have $(\beta,\beta) = \frac{3}{2}$ and $(\beta, -) = \frac{1}{2} H \in H^2(F, \BQ)$.}
together with property (b) of the virtual fundamental class, we find
\[\mathrm{ev}_*[M^0] = \mathrm{ev}_*[\Mbar_{0, 1}(F, \beta)] = \mathrm{ev}_*[\Mbar_{0, 1}(F, \beta)]^{\mathrm{vir}} = 60 H \in H^2(F, \mathbb{Q}).\]
To conclude $M^0 = D$, it suffices to prove the following proposition.

\begin{prop} \label{compo}
For a very general $F$, each $s \in S$ yields a distinct rational curve $\phi(p^{-1}(s)) \subset F$.
\end{prop}

\begin{proof}
Let $s_1, s_2 \in S$ be two distinct points and suppose
\[\phi(p^{-1}(s_1)) = \phi(p^{-1}(s_2)) \subset F.\]
For $i = 1, 2$, let $l_i \subset Y$ be the line corresponding to $s_i$, and let $P_i \subset \BP^5$ be the $3$-dimensional linear subspace spanned by the tangent planes along $l_i$. Then necessarily $P_1 = P_2$. Otherwise, the intersection $P_1 \cap P_2$ is a plane that contains all lines in $Y$ corresponding to the points on $\phi(p^{-1}(s_i))$. The fact that $Y$ contains a plane violates the very general assumption. We also know $l_1 \cap l_2 = \emptyset$. Otherwise, the plane spanned by $l_1$ and $l_2$ is tangent to $Y$ along both $l_1$ and $l_2$, which is impossible.
%\footnote{Let $l$ be a residual line of $l_1$. Then $l$ meets $l_1$. But by assumption $l$ is also a residual line to $l_2$ so meets $l_2$ and hence lies in the plane $P$ spanned by $l_1$ and $l_2$. Moreover, $P$ is also the plane spanned by $l_1$ and $l$ hence tangent to $l_1$. And similar $P$ is tangent to $l_2$, which is impossible.}

Consider the Gauss map\footnote{It is called the \emph{dual mapping} in \cite{GC}.} associated to the cubic $Y$,
\[\mathcal{D}: \BP^5 \to \BP^{5*}.\]
By definition, the image $\mathcal{D}(l_i) \subset \BP^{5*}$ is a line which is dual to $P_i \subset \BP^5$. Following the argument of Clemens and Griffiths \cite[Section 6]{GC}, we may assume that~$l_1, l_2$ are given by the equations
\begin{gather*}
X_2 = X_3 = X_4 = X_5 = 0,\\
X_0 = X_1 = X_4 = X_5 = 0.
\end{gather*}
Then the condition $P_1 = P_2$ forces $\mathcal{D}(l_1) = \mathcal{D}(l_2)$ to be given by the equations
\[X_0^* = X_1^* = X_2^* = X_3^* = 0.\]
As a result, the cubic polynomial of $Y$ takes the form
\begin{multline} \label{cubicpol}
X_4Q_4^1(X_0, X_1) + X_5Q_5^1(X_0, X_1) \\
+ X_4Q_4^2(X_2, X_3) + X_5Q_5^2(X_2, X_3) + R_1 + R_2.
\end{multline}
Here the $Q_i^j$ are quadratic polynomials, $R_1$ consists of terms of degree at least $2$ in $\{X_4, X_5\}$, and $R_2$ consists of terms of degree $1$ in each of $\{X_0, X_1\}, \{X_2, X_3\}, \{X_4, X_5\}$. The total number of possibly nonzero coefficients in \eqref{cubicpol} is
\[4 \cdot 3 + (4 \cdot 3 + 4) + 2 \cdot 2 \cdot 2 = 36.\]
On the other hand, the subgroup of $\mathrm{GL}(\mathbb{C}^6)$ fixing two disjoint lines in $\BP^5$ is of dimension
\[4 + 4 + 3 \cdot 4 = 20,\]
resulting in a locus of dimension $36 - 20 = 16$ in the moduli space of cubic $4$-folds. This again contradicts the very general assumption of $Y$.
\end{proof}

\subsection{Non-contribution}
We use the Gromov--Witten correspondence introduced in \eqref{GWcor} to eliminate the component $M^1$. Recall that by property~(b) of the virtual fundamental class, the class $[\Mbar_{0, 2}(F, \beta)]^{\mathrm{vir}}$ in \eqref{GWcor} equals the ordinary fundamental class. 

We begin by calculating the contribution of $M^0 = D$ to the Gromov--Witten correspondence
\begin{equation} \label{GWCorF}
\mathsf{GW}_\beta : H^4(F, \mathbb{Q}) \to H^4(F, \mathbb{Q}).
\end{equation}
Recall the diagram \eqref{uniruled} and consider morphisms
\begin{gather*}
\Phi_1^D = p_*\phi^* : H^4(F, \mathbb{Q}) \to H^2(S, \mathbb{Q}),\\
\Phi_2^D = \phi_*p^* : H^2(S, \mathbb{Q}) \to H^4(F, \mathbb{Q}).
\end{gather*}
Comparing with \eqref{341234} and \eqref{factorization}, we see that $\Phi_2^D \circ \Phi_1^D = \phi_*p^*p_*\phi^*$ gives the contribution of $D$ to the Gromov--Witten correspondence \eqref{GWCorF}.

Let $c = c_2(\mathcal{U}_F^*) \in H^4(F, \mathbb{Q})$. Using the short exact sequence
\[0 \to T_F \to T_{\mathrm{Gr}(2, 6)}|F \to \mathrm{Sym}^3(\mathcal{U}_F^*) \to 0,\]
we find
\[8c = 5H^2 - c_2(F) = v_F \in H^4(F, \mathbb{Q})\]
where $v_F$ is the class defined in \eqref{Markman_class}.\footnote{The proportionality of $c$ and $v_F$ also follows from the fact that $c$ is represented by a rational (hence Lagrangian) surface.} There is the following explicit calculation.

\begin{prop} \label{eigen}
We have
\[\phi_*p^*p_*\phi^*c = 945c \in H^4(F, \mathbb{Q}).\]
\end{prop}

\begin{proof}
The argument in Proposition \ref{prop2.4} shows that $c$ is an eigenvector of $\phi_*p^*p_*\phi^*$. To determine the eigenvalue, it suffices to compute the intersection number
\begin{equation} \label{intnum}
\int_F\phi_*p^*p_*\phi^*c \cdot H^2.
\end{equation}
By the projection formula, we have
\begin{multline*}
\int_F\phi_*p^*p_*\phi^*c \cdot H^2 = \int_D p^*p_*\phi^*c \cdot \phi^*H^2 \\
= \int_S p_*\phi^*c \cdot p_*\phi^*H^2 = \int_F\phi_*p^*p_*\phi^*H^2 \cdot c.
\end{multline*}
Again by the argument in Proposition \ref{prop2.4}, we know that $\phi_*p^*p_*\phi^*H^2$ is proportional to $c$. Hence we can deduce the intersection number \eqref{intnum} by calculating instead
\[\int_F\phi_*p^*p_*\phi^*H^2 \cdot H^2 = \int_S (p_*\phi^*H^2)^2.\]

Let $\xi$ be the relative hyperplane class of the projective bundle
\[p: D = \mathbb{P}(\mathcal{N}_{S/F}) \to S.\]
By \cite[Proposition 6]{A} and the projective bundle formula, we~find
\[p_*\phi^*H^2 = p_*(7p^*H_S + 3\xi)^2 = 42H_S - 9c_1(\mathcal{N}_{S/F}) \in H^2(S, \mathbb{Q}),\]
where $H_S$ is the restriction of $H$ to $S$. Moreover, Corollary~\ref{cor1.3} yields
\[c_1(\mathcal{N}_{S/F}) = -c_1(S) = 3H_S \in H^2(S, \mathbb{Q}).\]
Hence we obtain
\[p_*\phi^*H^2 = 15H_S \in H^2(S, \mathbb{Q}).\]
Applying Lemma \ref{calculation}, we find the intersection number
\[\int_F\phi_*p^*p_*\phi^*H^2 \cdot H^2 = \int_S (p_*\phi^*H^2)^2 = 15^2 \cdot 315 = 70875.\]

Finally, by the intersection numbers calculated in \cite[Lemma 4]{A}, we have
\[\int_F\phi_*p^*p_*\phi^*c \cdot H^2 = \int_F\phi_*p^*p_*\phi^*H^2 \cdot c = 70875 \cdot \frac{27}{45} = 42525\]
and hence
\[\phi_*p^*p_*\phi^*c = \frac{42525}{45}c = 945c \in H^4(F, \mathbb{Q}).\qedhere\]
\end{proof}

The eigenvalue in Proposition \ref{eigen} coincides with the one in Theorem~\ref{ThmGWb_1}, 
\[\mathsf{GW}_\beta(c) = 945 c \in H^4(F, \mathbb{Q}).\]
Hence the final step is to show that if the component $M^1$ is nonempty, then it has to contribute nontrivially to the Gromov--Witten correspondence \eqref{GWCorF}.

If $M' \subset M^1$ is a nonempty irreducible component, consider the restriction of \eqref{341234}
\[\begin{tikzcd}
M' \ar{r}{\ev} \ar{d}{p} & F \\
T'
\end{tikzcd}\]
where $T' \subset p(M^1) \subset \Mbar_{0, 0}(F, \beta)$ is the base of $M'$. We define morphisms
\begin{gather*}
\Phi_1^{M'} : H^4(F, \mathbb{Q}) \to H_2(T', \mathbb{Q}), \quad \gamma \mapsto p_\ast (\mathrm{ev}^\ast \gamma \cap [M']),\\
\Phi_2^{M'} = \mathrm{ev}_*p^* : H_2(T', \mathbb{Q}) \to H^4(F, \mathbb{Q}).
\end{gather*}
By definition, the composition $\Phi_2^{M'} \circ \Phi_1^{M'}$ gives the contribution of $M'$ to the Gromov--Witten correspondence \eqref{GWCorF}.

\begin{prop}
If $M' \subset M^1$ is a nonempty irreducible component, then we have
\[\Phi_2^{M'} \circ \Phi_1^{M'} (c) = Nc \in H^4(F, \mathbb{Q})\]
for some $N > 0$.
\end{prop}

\begin{proof}
Let $Z' = \mathrm{ev}(M')$ with $\iota: Z' \hookrightarrow F$ the embedding. Consider the following diagram
\[\begin{tikzcd}
\widetilde{M}' \ar{r}{\widetilde{\mathrm{ev}}} \ar{d}{\tau} & \widetilde{Z}' \ar{d}{} \ar{rd}{\tilde{\iota}} \\
M' \ar{r}{\ev} \ar{d}{p} & Z' \ar{r}{\iota} & F, \\
T'
\end{tikzcd}\]
where $\widetilde{M}'$ and $\widetilde{Z}'$ are simultaneous resolutions of $M'$ and $Z'$. 

We calculate $\Phi_1^{M'}(c) \in H_2(T', \mathbb{Q})$. By the projection formula, we~have\footnote{Since $\widetilde{M}'$ is nonsingular, we have suppressed an application of Poincar\'e duality here.}
\begin{align*}
\Phi_1^{M'}(c) & = p_*(\mathrm{ev}^*\iota^*c \cap[M']) \\
& = p_*\tau_*\tau^*\mathrm{ev}^*\iota^*c \\
& = p_*\tau_*\widetilde{\mathrm{ev}}^*\tilde{\iota}^*c \in H_2(T', \mathbb{Q}).
\end{align*}
Since $Z'$ is Lagrangian, we find
\[[Z'] = \tilde{\iota}_*[\widetilde{Z}'] = N'c \in H^4(F, \BQ)\]
for some $N' > 0$. The intersection number $\int_Fc^2 = 27$ calculated in \cite[Lemma 4]{A} then implies
\[\tilde{\iota}^*c = 27N'[\tilde{x}] \in H^4(\widetilde{Z}', \mathbb{Q})\]
for any point $\tilde{x} \in \widetilde{Z}'$. This yields
\[\Phi_1^{M'}(c) = 27N'p_*\tau_*\widetilde{\mathrm{ev}}^*[\tilde{x}] = 27N'[V_x] \in H_2(T', \mathbb{Q}),\]
where $V_x \subset T'$ parametrizes rational curves through a general point $x \in Z'$. In particular, we see that $\Phi_1^{M'}(c) \in H_2(T', \mathbb{Q})$ is an effective curve class.

As a result, the class
\[\Phi_2^{M'} \circ \Phi_1^{M'}(c) = \mathrm{ev}_*p^*  \Phi_1^{M'}(c) \in H^4(F, \mathbb{Q})\]
is an effective sum of classes of Lagrangian surfaces, and hence a positive multiple of $c$.
\end{proof}

We conclude $M^1 = \emptyset$, and the proof of Theorem \ref{mainthm} is complete.

\appendix
\section{Sketch of a classical proof of Theorem \ref{mainthm}} \label{Classical}
We sketch a proof of Theorem \ref{mainthm} via the classical geometry of cubic hypersurfaces. Let $Y \subset \BP^5$ be a very general cubic $4$-fold, and let $F$ be the Fano variety of lines in $Y$.

Consider the correspondence given by the universal family
\begin{equation*}
\begin{tikzcd}
\mathcal{L}\arrow{r}{q_Y} \arrow{d}{q_F} & Y \\
F.
\end{tikzcd} 
\end{equation*}
A rational curve $R \subset F$ corresponds to a surface $Z = q_Y(q_F^{-1}(R)) \subset Y$. If~$R$ lies in the primitive curve class of $F$, then we have 
\[
[Z] = H_Y^2 \in H^4(Y, \BZ)
\]
with $H_Y$ the hyperplane class on $Y$.

\medskip
\noindent {\emph{Step 1.}} Let $j: Y \hookrightarrow \BP^5$ be the embedding. Since the surface $j(Z) \subset \BP^5$ is of degree $3$, we know from \cite[Page 173]{GH} that $j(Z)$ lies in a hyperplane $\BP^4 \subset \BP^5$. Hence $Z$ is contained in the hyperplane section
\[
Y' = Y \cap \BP^4 \subset \BP^4.
\]

\medskip
\noindent {\emph{Step 2.}} By \cite[Page 525, Proposition]{GH}, the surface $Z \subset Y'$ belongs to one of the following classes:
\begin{enumerate}
    \item[(i)] a cubic rational normal scroll;
    \item[(ii)] a cone over a twisted cubic curve;
    \item[(iii)] a cubic surface given by a hyperplane section of $Y'\subset \BP^4$.
\end{enumerate}

Since (i) and (ii) cannot hold for a very general\footnote{Case (i) corresponds to the divisor $\mathcal{C}_{12}$ in the moduli space of cubic $4$-folds; see \cite{Ha}. Case (ii) is a degeneration of (i), and can be argued by a similar dimension count.} cubic $4$-fold, we find that $Z$ is a cubic surface of the form
\[
Z = Y \cap \BP^3.
\]

\medskip
\noindent {\emph{Step 3.}} The singularities of cubic surfaces were classified long ago; see \cite[Chapter 9]{Do} and \cite[Section 2]{LLSV}. Since $Z$ is integral, it satisfies one of the following conditions:
\begin{enumerate}
    \item[(i)] $Z$ has rational double point singularities;
    \item[(ii)] $Z$ has a simple elliptic singularity;
    \item[(iii)] $Z$ is integral but not normal.
\end{enumerate}

By definition, the surface $Z$ is swept out by a $1$-dimensional family of lines parameterized by a rational curve. Hence we may narrow down to case~(iii).

\medskip
\noindent {\emph{Step 4.}} By further classification results (see \cite[Section 2.3]{LLSV}), the surface $Z$ is projectively equivalent to one of the four surfaces with explicit equations:
\begin{gather*}
X_0^2X_1 + X_2^2X_3 =0,\\
X_0X_1X_2 + X_0^2X_3 + X_1^3 = 0,\\
X_1^3 + X_2^3 + X_1X_2X_3 =0,\\
X_1^3 + X_2^2X_3 = 0.
\end{gather*}
In each of the four cases, the singular locus of $Z$ is a line $l \subset Z$, and the $1$-dimensional family of lines covering $Z$ is given by the residual lines of the planes containing $l$. Hence we conclude that all rational curves in the primitive curve class of $F$ are given by the uniruled divisor \eqref{uniruled}. The uniqueness part of Theorem \ref{mainthm} follows from Proposition \ref{compo}.

\end{document}